\documentclass[12pt]{article}
\usepackage[latin1]{inputenc}

\usepackage{enumitem}

\usepackage{amsmath,amsthm,amssymb,amsfonts}
\usepackage{mathtools}
\usepackage[pdftex]{graphicx}
\usepackage{graphicx}
\usepackage{epsfig}
\usepackage{verbatim}
\usepackage{footmisc}
\usepackage{bbm,bm}
\usepackage{lscape}
\usepackage[longnamesfirst]{natbib}
\usepackage{color}
\usepackage{float}
\usepackage{longtable}
\usepackage{lipsum}

\usepackage{booktabs}

\usepackage{natbib}

\newtheorem{Theorem}{Theorem}[section] 
\newtheorem{Lemma}[Theorem]{Lemma} 
\newtheorem{Remark}[Theorem]{Remark} 
\newtheorem{Corollary}[Theorem]{Corollary}

\makeatletter

\@addtoreset{equation}{section}
\makeatother

\renewcommand{\baselinestretch}{1.2}

\newcommand{\R}{\mathbb{R}}
\newcommand{\N}{\mathbb{N}}
\newcommand{\Z}{\mathbb{Z}}

\newcommand{\cF}{\mathcal{F}}

\newcommand{\cC}{\mathcal{C}}
\newcommand{\cD}{\mathcal{D}}
\newcommand{\cH}{\mathcal{H}}

\newcommand{\Var}{\operatorname{Var}}

\newcommand{\si}{\sigma}
\newcommand{\ep}{\varepsilon}
\newcommand{\om}{\omega}

\newcommand{\lf}{\left\lfloor}
\newcommand{\rf}{\right\rfloor}
\newcommand{\mi}{\wedge}

\newcommand{\ph}{\varphi}

\allowdisplaybreaks

\usepackage{color}

\usepackage{hyperref}

\begin{document}

\title{\bf Nonparametric volatility change detection}

\author{{\sc Maria Mohr} and {\sc Natalie Neumeyer}\\ Department of Mathematics, University of Hamburg}

\maketitle

\renewcommand{\baselinestretch}{1.1}

\begin{abstract}
We consider a nonparametric heteroscedastic time series regression model and suggest testing procedures to detect changes in the conditional variance function. The tests are based on a sequential marked empirical process and thus combine classical CUSUM tests with marked empirical process approaches known from goodness-of-fit testing. The tests are consistent against general alternatives of a change in the conditional variance function, a feature that classical CUSUM tests are lacking. We derive a simple limiting distribution and in the case of univariate covariates even obtain asymptotically distribution-free tests. We demonstrate the good performance of the tests in a simulation study and consider exchange rate data as a real data application. 
\end{abstract}

\vspace*{.5cm}






\small
\normalsize

\section{Introduction}  \label{introduction}


The paper is concerned with the investigation of structural stability of the conditional variance function (volatility function) in nonparametric heteroscedastic time series regression models.
Those models have gained much attention over the last decades and contain as special cases nonparametric AR-ARCH models, which are also called nonparametric CHARN (conditional heteroscedastic autoregressive nonlinear) models; see \cite{Fan2003}
 or \cite{Gao2007}
 for overviews.  
They have been successfully applied to model econometric time series such as  foreign exchange rates or stock market indices, see e.g.\ \cite{Yang1999579} and \cite{ZhaoWu}.
 Here tests for structural changes in the volatility function are of special importance.  

A lot of research has been devoted to the parametric case, notably for ARCH and GARCH models. Among others, \cite{Kokoszka1999182} suggested a CUSUM type test for parameter stability in ARCH models, while \cite{Kulperger20052395} considered partial sums of higher powers of residuals to test for a parameter change in GARCH models. \cite{Berkes2004263} considered tests for parameter stability in GARCH models based on likelihood ratios. \citeauthor{Kengne2012503}'s (\citeyear{Kengne2012503}) test, which is based on quasi likelihood estimators, is applicable to more general parametric causal time series models. \cite{Lee2014101} suggested a residual based CUSUM test for change points in parametric AR-GARCH models, while \cite{Lee20081990} and \cite{Song201841} considered ARMA-GARCH models. Very few results are available in the nonparametric framework. \cite{Chen200579} studied a nonparametric heteroscedastic time series model with a scale change in volatility. However, they assume a compact support of regressors, which is problematic when considering autoregression models.
Tests for change points in the unconditional variance in time series models have been considered as well. 
\cite{Lee2003467} considered parametric autoregression models, as well as fixed design nonparametric regression models with strongly mixing errors using a CUSUM testing procedure. \cite{Chen20141} constructed a ratio test for change point detection in the variance in random design nonparametric regression models. However, their test does not allow for autoregressive effects, as a compact support of regressors is assumed.
A related strand of the literature deals with change point detection in the error distribution of a time series regression model. In the parametric framework \cite{Koul1996380} considered non-linear regression models and \cite{Ling1998741} non-stationary AR models, to just mention a few,
while \cite{Selk2013770} considered nonparametric heteroscedastic autoregression models.

Recently, \cite{Mohr2019} suggested a test for change point in the regression function in nonparametric time series models. They combine traditional CUSUM tests as considered by \cite{Hidalgo1995671}, \cite{Honda199745} and \cite{Su2008347} in the nonparametric context with the marked empirical process approach originally suggested by \cite{Stute1997613} and widely used in the goodness-of-fit literature. Compared with the CUSUM approach the new test shows better power properties, in theory as well as in finite sample simulations. 
In the paper at hand we will modify the CUSUM marked empirical process test in order to test for a change point in the conditional volatility function. 
We obtain tests with  very simple limiting distributions, which are consistent against general fixed alternatives. In the case of univariate covariates one can even obtain tests that are asymptotically distribution-free. 

The paper is organized as follows. In section 2 we define the process on which the test statistics are built. In section 3 we give the limiting distribution of the process under the null hypothesis of no change in the variance function. We further discuss consistency against fixed alternatives of one change point. In section 4 we describe a simulation study and discuss a real data example of currency exchange rates. Section 5 concludes the paper, whereas in the appendix we list the regularity assumptions and prove the asymptotic results. 

\section{The model and test statistic}\label{sec:model}

Consider a strictly stationary and strongly mixing time series $(Y_t,\bm{X}_t)$, $t\in\Z$,  following the nonparametric model
\begin{equation}\label{model}
Y_t=m(\bm{X}_t)+U_t, 
\end{equation}
where $E[U_t|\cF^t]=0$ a.s.\ for the sigma-field $\cF^{t}=\sigma(U_{j-1},\bm{X}_{j}:j\le t)$, and $m:\R^d\to \R$ does not depend on $t$. 
Further, let the following representation for the innovations $U_t$ hold,
\begin{equation}\label{model2}
U_t=\si_t(\bm{X}_t)\ep_t, \ t\in\Z,
\end{equation}
for some functions $\si_t:\R^d\to\R$ and an i.i.d. sequence $(\ep_t)_{t\in\Z}$, such that $\ep_t$ is independent of $\bm{X}_j$ for all $j\le t$ and fulfills $E[\ep_1]=0$, $E[\ep_1^2]=1$ and $E[\ep_1^4]<\infty$. With these restrictions, $\si_t^2$ is the  variance function of $Y_t$, conditioned on $\bm{X}_t$, as
\[\Var(Y_t|\bm{X}_t)=E[U_t^2|\bm{X}_t]=\si_t^2(\bm{X}_t)\mbox{ a.s.}\] 
The $d$-dimensional absolutely continuous covariate $\bm{X}_t$ may include finitely many lagged values of $Y_t$, for instance $\bm{X}_t=(Y_{t-1},\dots,Y_{t-d})^T$, such that the model includes nonparametric AR-ARCH models.  

Our aim is to test whether the function $\si_t^2(\cdot)$ is stable in time $t$. Given observations $(Y_1,\bm{X}_1),\dots,(Y_n,\bm{X}_n)$ the null hypothesis
\[{H}_0: \si_t^2(\cdot)=\si^2(\cdot), \ t=1,\dots,n,\]
for some not further specified function $\si^2:\R^d\to \R$ (not depending on time $t$) will be considered.

The idea is to base tests for $H_0$ on a sequential marked empirical process of residuals, 
\begin{equation}\label{hattildeT}
\hat{ T }_n(s,\bm{z})=\frac{1}{\sqrt{n}}\sum\limits_{t=1}^{\lf ns \rf}\left((Y_t-\hat{m}_n(\bm{X}_t))^2-\hat{\si}^2_n(\bm{X}_t)\right)\om_n(\bm{X}_t)I\{\bm{X}_t\le \bm{z}\}
\end{equation}
indexed in $s\in[0,1]$ and $\bm{z}\in\R^d$. 
Throughout $I\{\dots\}$ denotes an indicator function. 
Further $\om_n(\cdot)=I\{\cdot \in \bm{J}_n\}$ is a weight function with $\bm{J}_n$ specified in assumption \textbf{(J)} in appendix A. The regression and volatility functions are estimated as 
\begin{eqnarray*}
\hat{m}_n(\bm{x})&=& \dfrac{\sum_{j=1}^{n}K\Big(\frac{\bm{x}-\bm{X}_{j}}{h_n}\Big)Y_{j}}{\sum_{j=1}^{n}K\Big(\frac{\bm{x}-\bm{X}_{j}}{h_n}\Big)}
\end{eqnarray*}
and
\begin{eqnarray*}
\hat{\si}_n^2(\bm{x})&=&\frac{\sum_{j=1}^{n}K\Big(\frac{\bm{x}-\bm{X}_{j}}{h_n}\Big)(Y_j-\hat{m}_n(\bm{x}))^2}{\sum_{j=1}^{n}K\Big(\frac{\bm{x}-\bm{X}_{j}}{h_n}\Big)},
\end{eqnarray*}
respectively, with kernel function $K$ and bandwidth $h_n$ as considered in the assumptions in appendix A. 
The null hypothesis $H_0$ of no change in the variance will be rejected for large values of, e.g., a  Kolmogorov-Smirnov type test statistic
\begin{align*}
 T _{n1}:=\sup\limits_{\bm{z}\in\R^d}\sup\limits_{s\in[0,1]}\left|\hat{ T }_n(s,\bm{z})\right|
\end{align*}
due to the following motivation. 
Note  that the volatility function $\si^2_t$ from (\ref{model2}) can be viewed as  regression function in a  regression model
$$U_t^2=\si^2_t(\bm{X}_t)+\xi_t, \quad t\in\Z,$$
with covariate $\bm{X}_t$, response variable $U_t^2$ and innovations
$\xi_t=U_t^2-\si^2_t(\bm{X}_t)$, 
that satisfy
$E[\xi_t|\bm{X}_t]=0$ and $E[\xi_t^2|\bm{X}_t]=\si_t^4(\bm{X}_t)E[(\ep^2_t-1)^2]$ 
a.s. However, this is not a feasible model as $U_t=Y_t-m(\bm{X}_t)$ is unobservable and has to be estimated. 
The term 
$$(Y_t-\hat{m}_n(\bm{X}_t))^2-\hat{\si}^2_n(\bm{X}_t)=:\hat\xi_t$$
 in the definition of the process $\hat{ T }_n$ can be seen as estimator for the innovation $\xi_t$ in the `non-feasible' model above under the null hypothesis $\si_t^2(\cdot)=\si^2(\cdot)\forall t$. Thus $n^{-1/2}\hat{ T }_n$ will vanish for $n\to\infty$ under the null hypothesis. The limiting process of $\hat T_n$ will be given in Corollary \ref{cor:T_dach} below. From this result critical values for a test based on the Kolmogorov-Smirnov type test statistic $ T _{n1}$ can be approximated. The behavior of  $ T _{n1}$ under fixed alternatives will be demonstrated in Remark \ref{consistency} in order to motivate consistency of the test. The process $\hat{ T }_n$ is a consistent improvement of CUSUM tests analogous to the procedure in \cite{Mohr2019} developed for changes in the regression function.  

\begin{Remark}
In model (\ref{model}) we assume a regression function $m$ that is stable in time $t$. For testing of a change in the variance function this assumption makes sense if beforehand one can test for a change in the regression function applying a testing procedure which only reacts sensitive to changes in the regression function, not to changes in the variance function. \cite{Mohr2019} provide such a bootstrap test, which can be applied in cases of unstable variances, but as desired only reacts sensitive to changes in the regression function. Consecutively applying the bootstrap test in \cite{Mohr2019} and, if it does not reject, the test in the paper at hand, gives the knowledge of whether a change occurs in the mean or the variance function. 
\end{Remark}

\section{Asymptotic results} 

Under the regularity assumptions in appendix A one can derive the following decomposition of the process $\hat T_n$ defined in (\ref{hattildeT}) in terms of the process
\begin{equation}\label{tildeT}
 T_n(s,\bm{z})=\frac{1}{\sqrt{n}}\sum\limits_{t=1}^{\lf ns \rf}\xi_t I\{\bm{X}_t\le\bm{z}\},\quad s\in [0,1],z\in\R^d,
\end{equation}
as well as the weak convergence of  $ T_n$. 

\begin{Theorem} \label{mainth}
Assume model (\ref{model}), (\ref{model2}) under the null hypothesis $H_0$ and assumptions \textbf{(G)}, \textbf{($\bm{\xi}$)}, \textbf{(M)}, \textbf{(J)}, \textbf{(F1)}, \textbf{(F2)}, \textbf{(K)}, \textbf{(B1)} and \textbf{(B2)} from appendix A.  

{\bf (i)} Then, $\hat{ T }_n(s,\bm{z})= T _n(s,\bm{z})-s T _{n}(1,\bm{z})+o_P(1)$
uniformly in $s\in[0,1]$ and $\bm{z}\in\R^d$. 

{\bf (ii)} The process $ T _n=\{ T _n(s,\bm{z}):s\in[0,1],\bm{z}\in\R^d\}$
converges weakly in $\ell^{\infty}([0,1]\times\R^d)$ to a centered Gaussian process ${G}$ with
%
\[\mbox{\rm Cov}\left({G}(s_1,\bm{z}_1),{G}(s_2,\bm{z}_2)\right)=(s_1\mi s_2){\Sigma}(\bm{z}_1\mi \bm{z}_2),\]
where ${\Sigma}(\bm{z}):=E[(\ep^2_1-1)^2]\int_{(-\bm{\infty},\bm{z}]}\si^4(\bm{x})f(\bm{x})d\bm{x}$.
\end{Theorem}
\noindent Here and throughout we define $(-\bm{\infty},\bm{z}]=(-\infty,z_1]\times\cdots\times(-\infty,z_d]$ for $\bm{z}=(z_1,\dots,z_d)\in\R^d$. 
The proof of Theorem \ref{mainth} is given in appendix B. An application of the continuous mapping theorem and Slutsky's lemma give the following weak convergence result for the process $\hat{ T }_n$.

\begin{Corollary}\label{cor:T_dach}
Suppose that the assumptions of Theorem \ref{mainth} and $H_0$ are satisfied. Then the process 
$\hat{ T }_n$ converges weakly in $\ell^{\infty}([0,1]\times\R^d)$ to a centered Gaussian process ${G}_0$ with
%
\[\mbox{\rm Cov}\left({G}_0(s_1,\bm{z}_1),{G}_0(s_2,\bm{z}_2)\right)=(s_1\mi s_2-s_1s_2){\Sigma}(\bm{z}_1\mi \bm{z}_2).\]
\end{Corollary}

\noindent The continuous mapping theorem then implies convergence in distribution of the Kolmogorov-Smirnov test statistic, 
\[ T _{n1}\underset{n\to\infty}{\overset{\cD}{\to}}\sup\limits_{\bm{z}\in\R^d}\sup\limits_{s\in[0,1]}\left|{G}_0(s,\bm{z})\right|.\]
In particular in the case $d=1$ using continuity of ${\Sigma}$ and the scaling property of the Brownian motion, it holds that
$ T _{n1}$ converges in distribution to $c^{1/2}T$, where $$T=\sup_{s\in[0,1]}\sup_{t\in[0,1]}|K_0(s,t)|$$ and  $K_0$ is  a Kiefer-M\"{u}ller process.
The constant ${c}=E[((Y_1-m(X_1))^2-\si^2(X_1))^2]$ can be consistently estimated as 
\[\hat{{c}}_n:=\frac{1}{n}\sum\limits_{i=1}^{n}\left((Y_i-\hat{m}_n(X_i))^2-\hat{\si}^2_n(X_i)\right)^2\om_n(X_i),\] 
and the test statistic $ T _{n1}/\hat{{c}}_n^{1/2}$ is asymptotically distribution-free. 
We reject $H_0$ at asymptotic level $\alpha$ if $ T _{n1}/\hat{{c}}_n^{1/2}$  is larger than the (known) $(1-\alpha)$-quantile of $T$. 

\begin{Remark}\label{consistency} To see that the test is consistent against simple fixed alternatives of one change in the volatility function, 
\begin{align*}
{H}_{1}: \exists s_0\in (0,1): \si^2_{n,t}(\cdot)=\begin{cases}\si^2_{(1)}(\cdot), & t=1,\dots, \lf ns_0\rf \\ \si^2_{(2)}(\cdot), & t=\lf ns_0\rf+1,\dots,n,\end{cases} 
\end{align*}
for some functions with $\si^2_{(1)}\not\equiv \si^2_{(2)}$, consider a triangular array
\begin{align*} 
Y_{n,t}=m(\bm{X}_{n,t})+U_{n,t}, \ t=1,\dots,n,
\end{align*}
with regression function $m$ stable in time and innovations such that $E[U_{n,t}|\cF_n^t]=0$ and $E[U_{n,t}^2|\bm{X}_{n,t}]=\si_{n,t}^2(\bm{X}_{n,t})$ a.s. Further assume that the covariate $\bm{X}_{n,t}$ is absolutely continuous with density function $f_{n,t}$. Then $\hat\sigma_n^2(\bm{x})$ will estimate the function
\begin{align*}
\bar{\si}^2_n(\bm{x})
&=\frac{\frac{1}{n}\sum\limits_{i=1}^{n}f_{n,i}(\bm{x})\si^2_{n,i}(\bm{x})}{\frac{1}{n}\sum\limits_{i=1}^{n}f_{n,i}(\bm{x})}=\left(\si^2_{(1)}(\bm{x})-\si^2_{(2)}(\bm{x})\right)\frac{\frac{1}{n}\sum\limits_{i=1}^{\lf ns_0 \rf}f_{n,i}(\bm{x})}{\frac{1}{n}\sum\limits_{i=1}^{n}f_{n,i}(\bm{x})}+\si^2_{(2)}(\bm{x}).
\end{align*}
Now assume that for each $s\in (0,1)$, the limit of $n^{-1}\sum_{i=1}^{\lf ns \rf}f_{n,i}$ exists and denote it by $\bar f^{(s)}$. Then $n^{-1/2}\hat{ T }_n(s_0,\bm{z})$ will converge in probability to the integral
$$\int\limits_{(-\bm{\infty},\bm{z}]}\left(\si^2_{(1)}(\bm{u})-\si^2_{(2)}(\bm{u})\right)\bar{f}^{(s_0)}(\bm{u})\left(1-\frac{\bar{f}^{(s_0)}(\bm{u})}{\bar{f}^{(1)}(\bm{u})}\right)d\bm{u},$$
which, under $H_1$, does not vanish for at least one $\bm{z}=\bm{z}_0$ (provided that $\bar f^{(s_0)}\neq \bar f^{(1)}$). As $ T_{n1}\geq |\hat{ T }_n(s_0,\bm{z}_0)|$, the test statistic will converge to infinity in probability and the test is consistent. 
\end{Remark}

\begin{Remark}\label{consistency-CUSUM}
A traditional CUSUM test statistic in our context would be defined as $\sup_{s\in[0,1]}|\hat{T}_n(s,\infty)|$. With the same reasoning as in Remark \ref{consistency},   $n^{-1/2}\hat{ T }_n(s_0,\infty)$ will converge in probability to 
$$\int\left(\si^2_{(1)}(\bm{u})-\si^2_{(2)}(\bm{u})\right)\bar{f}^{(s_0)}(\bm{u})\left(1-\frac{\bar{f}^{(s_0)}(\bm{u})}{\bar{f}^{(1)}(\bm{u})}\right)d\bm{u},$$
which could be zero, even under the alternative $H_1$. In such a case the CUSUM test is not consistent. 
\end{Remark}

\section{Finite sample properties}  

\subsection{Simulations}

A Monte Carlo study is conducted in order to compare the results for $T_{n1}$ from section \ref{sec:model} and a Cram\'er-von Mises type test $T_{n2}:=\sup_{z\in\R}\int_{0}^{1}|\hat{T}_n(s,z)|^2ds$ with those of the traditional CUSUM versions denoted by $KS:=\sup_{s\in[0,1]}|\hat{T}_n(s,\infty)|$ and $CM:=\int|\hat{T}_n(s,\infty)|^2ds$. All simulations are carried out with a level of $5\%$, $1000$ replications and for sample sizes $n\in\{100,300,500\}$. For the nonparametric estimators $\hat{m}_n$ and $\hat{\si}^2_n$ we use an Epanechnikov kernel $K$ and $h_n=n^{-1/3}$ as a simple ad hoc bandwidth. Furthermore, we set $c_n=\log(n)$ for the weighting function. The data is simulated from the following models.
\begin{alignat*}{3}
&&\text{(model 1)} 
\hspace{2cm}&Y_t=m(X_t)+\si_t(X_t)\ep_t, \ \ep_t\sim\mathcal{N}(0,1), & \\
&~&&\si_t(x)=\begin{cases}0.5\exp(-0.2x), \ & t=1,\dots, \lf ns_0\rf \\0.5\exp(0.2x), \ & t=\lf ns_0 \rf +1,\dots, n,\end{cases}
\end{alignat*}
where $X_t$ is an exogenous variable following the AR(1) model $X_t=0.4X_{t-1}+\xi_t$ with $\xi_t$ being i.i.d.~$\sim\mathcal{N}(0,1)$.
\begin{alignat*}{3}
&&\text{(model 2)} 
\hspace{2cm}&Y_t=m(Y_{t-1})+\si(Y_{t-1})\ep_t,  \ \ep_t\sim\mathcal{N}(0,1),& \\
&~&&\si_t(x)=\begin{cases}\sqrt{0.1+0.1x^2}, \ & t=1,\dots, \lf ns_0\rf \\\sqrt{0.1+0.7x^2}, \ & t=\lf ns_0 \rf +1,\dots, n.\end{cases} \hspace{0.3cm}
\end{alignat*}
For both model 1 and 2 we  consider $s_0\in \{0,0.25,0.5,0.75,1\}$ and two different choices for the regression function, namely $m(x)=0.5x$ (case (a)) and $m(x)=-0.5x$ (case (b)).

Model 1 is a heteroscedastic regression model with autoregressive covariables while model 2 is a heteroscedastic autoregression (AR-ARCH) model. In both cases $H_0$ is satisfied for $s_0\in\{0,1\}$ and $H_1$ is satisfied for $s_0\in\{0.25,0.5,0.75\}$. Further, note that data generated from both models fulfill the stationarity and mixing assumptions when $s_0\in\{0,1\}$ (see Remark \ref{rem G} in appendix A). 

Table \ref{sim_model 1} shows the rejection frequencies for model 1. To summarize the performance of the tests it is to mention that all level simulations ($s_0\in\{0,1\}$) show reasonably good results. The tests based on $T_{n1}$ and $T_{n2}$ show nice consistency properties ($s_0\in\{0.25,0.5,0.75\}$), rejecting the null more frequently with increasing sample sizes, where $T_{n2}$ has larger power. The classical CUSUM tests, however, clearly fail in detecting the change, having a power that does not exceed $10 \%$ for all cases (see Remark \ref{consistency-CUSUM}). All of the tests perform rather poorly when the sample size is small, i.e.\ for $n=100$. Furthermore, we note that changes occurring at $s_0=0.5$ are easiest to detect.
\begin{table}[htbp]
\small
  \centering 
	\caption{Rejection frequencies in model 1}\label{sim_model 1}
\begin{tabular}{@{}r r c r r r r c r r r r @{}} 
    \toprule
\multicolumn{2}{c}{}& \phantom{a} & \multicolumn{4}{c}{model 1 (a)} & \phantom{a} & \multicolumn{4}{c}{model 1 (b)}\\
\cmidrule{4-7} \cmidrule{9-12}
 $s_0$& $n$  &  & $T_{n1}$   & $T_{n2}$  &  $KS$ & $CM$  &  & $T_{n1}$   & $T_{n2}$  &  $KS$ & $CM$\\ 
\midrule 
$0$   & $100$  & &  		 $0.035$  &  $0.058$  &	 $0.046$  &  $0.042$  	& &		$0.052$  &  $0.064$  &  $0.059$  &  $0.047$	 \\
			   & $300$  & &      $0.053$  &  $0.073$  &  $0.058$  &  $0.048$    & &   $0.056$  &  $0.068$  &  $0.064$  &  $0.055$  \\		
				 & $500$  & &      $0.057$  &  $0.071$  &  $0.062$  &  $0.055$    & &   $0.062$  &  $0.064$  &  $0.059$  &  $0.043$  \\
\midrule 
$0.25$   & $100$  & &  		 $0.041$  &  $0.080$  &	 $0.052$  &  $0.048$  	& &		$0.053$  &  $0.081$  &  $0.060$  &  $0.054$	 \\
			   & $300$  & &      $0.112$  &  $0.157$  &  $0.072$  &  $0.062$    & &   $0.099$  &  $0.155$  &  $0.072$  &  $0.051$  \\		
				 & $500$  & &      $0.187$  &  $0.266$  &  $0.069$  &  $0.050$    & &   $0.216$  &  $0.294$  &  $0.097$  &  $0.080$  \\
\midrule 
$0.50$   & $100$  & &  		$0.063$  &  $0.122$  &  $0.053$  &  $0.057$			&	&	 	$0.068$  &  $0.120$  &  $0.073$  &  $0.066$  \\
			   & $300$  & &  		$0.210$  &  $0.276$  &  $0.091$  &  $0.073$   	&	&   $0.199$  &  $0.279$  &  $0.081$  &  $0.068$  \\		
				 & $500$  & &  		$0.413$  &  $0.521$  &  $0.097$  &  $0.067$			& &	  $0.428$  &  $0.510$  &  $0.096$  &  $0.074$ 	\\		
\midrule 
$0.75$   & $100$  & &  		$0.055$  &  $0.092$  &  $0.074$  &  $0.067$			&	&  	$0.051$  &  $0.084$  &  $0.061$  &  $0.066$  \\
			   & $300$  & &  		$0.130$  &  $0.174$  &  $0.079$  &  $0.053$		  &	&   $0.120$  &  $0.196$  &  $0.080$  &  $0.068$  \\		
				 & $500$  & &  		$0.222$  &  $0.291$  &  $0.086$  &  $0.074$			&	&   $0.239$  &  $0.304$  &  $0.096$  &  $0.074$  \\		
						\midrule 
$1$      & $100$  & &  		 $0.045$  &  $0.072$  &	 $0.062$  &  $0.057$  	& &		$0.046$  &  $0.076$  &  $0.055$  &  $0.055$	 \\
			   & $300$  & &      $0.053$  &  $0.064$  &  $0.056$  &  $0.041$    & &   $0.076$  &  $0.088$  &  $0.081$  &  $0.064$  \\		
				 & $500$  & &      $0.063$  &  $0.071$  &  $0.072$  &  $0.050$    & &   $0.064$  &  $0.070$  &  $0.064$  &  $0.051$  \\
\bottomrule  
\end{tabular}%
\end{table}

The corresponding results in model 2 can be found in table \ref{sim_model 2}. The level of $5\%$ is approximately hold for all tests, even in the case where the variance has a relatively large influence ($s_0=0$). The power simulations suggest that our tests as well as the classical CUSUM tests result in reasonable rejection probabilities, detecting the change more often for increasing sample sizes. Again changes in $s_0=0.5$ are easiest to detect. 

\begin{table}[htbp]
\small
  \centering 
	\caption{Rejection frequencies in model 2}\label{sim_model 2}
\begin{tabular}{@{}r r c r r r r c r r r r @{}} 
    \toprule
\multicolumn{2}{c}{}& \phantom{a} & \multicolumn{4}{c}{model 2 (a)} & \phantom{a} & \multicolumn{4}{c}{model 2 (b)}\\
\cmidrule{4-7} \cmidrule{9-12}
 $s_0$& $n$  &  & $T_{n1}$   & $T_{n2}$  &  $KS$ & $CM$  &  & $T_{n1}$   & $T_{n2}$  &  $KS$ & $CM$\\ 
\midrule 
$0$   & $100$  & &  		 $0.036$  &  $0.066$  &	 $0.041$  &  $0.045$  	& &		$0.039$  &  $0.070$  &	 $0.046$  &  $0.045$ \\

			   & $300$  & &      $0.056$  &  $0.066$  &  $0.062$  &  $0.041$    & &   $0.040$  &  $0.054$  &  $0.046$  &  $0.045$  \\		
				 & $500$  & &      $0.059$  &  $0.064$  &  $0.070$  &  $0.056$    & &   $0.059$  &  $0.074$  &  $0.066$  &  $0.064$  \\
\midrule 
$0.25$   & $100$  & &  		 $0.054$  &  $0.094$  &	 $0.086$  &  $0.079$  	& &		$0.068$  &  $0.096$  &  $0.080$  &  $0.081$	 \\
			   & $300$  & &      $0.165$  &  $0.209$  &  $0.218$  &  $0.214$    & &   $0.153$  &  $0.202$  &  $0.200$  &  $0.194$  \\		
				 & $500$  & &      $0.317$  &  $0.365$  &  $0.405$  &  $0.376$    & &   $0.274$  &  $0.338$  &  $0.364$  &  $0.350$  \\
\midrule 
$0.50$   & $100$  & &  		$0.086$  &  $0.134$  &  $0.123$  &  $0.137$			&	&	 	$0.100$  &  $0.141$  &  $0.143$  &  $0.142$  \\
			   & $300$  & &  		$0.414$  &  $0.433$  &  $0.507$  &  $0.470$   	&	&   $0.423$  &  $0.438$  &  $0.510$  &  $0.470$  \\		
				 & $500$  & &  		$0.743$  &  $0.746$  &  $0.829$  &  $0.780$			& &	  $0.748$  &  $0.746$  &  $0.809$  &  $0.782$ 	\\		
\midrule 
$0.75$   & $100$  & &  		$0.076$  &  $0.110$  &  $0.109$  &  $0.115$			&	&  	$0.082$  &  $0.128$  &  $0.115$  &  $0.119$  \\
			   & $300$  & &  		$0.329$  &  $0.361$  &  $0.410$  &  $0.376$		  &	&   $0.340$  &  $0.353$  &  $0.402$  &  $0.368$  \\		
				 & $500$  & &  		$0.655$  &  $0.636$  &  $0.724$  &  $0.667$			&	&   $0.631$  &  $0.614$  &  $0.705$  &  $0.651$  \\		
						\midrule 
$1$      & $100$  & &  		 $0.049$  &  $0.065$  &	 $0.054$  &  $0.050$  	& &		$0.044$  &  $0.082$  &  $0.053$  &  $0.048$	 \\
			   & $300$  & &      $0.069$  &  $0.063$  &  $0.068$  &  $0.054$    & &   $0.054$  &  $0.073$  &  $0.063$  &  $0.051$  \\		
				 & $500$  & &      $0.064$  &  $0.075$  &  $0.081$  &  $0.052$    & &   $0.056$  &  $0.065$  &  $0.061$  &  $0.045$  \\
\bottomrule  
\end{tabular}%
\end{table}

\subsection{Data example}

In this section we will apply our test to a financial data set that is concerned with exchange rates of currencies. Exchange rate regimes indicate how a country manages its currency with respect to other currencies, it can 
vary from "fixed", over "pegged" to "floating". In the case of a fixed regime, the currency is more or less fixed to some other currency. Contrarily with a floating regime the currency is allowed to fluctuate freely by market forces. Pegged regimes are somehow in between, the currency then has limited flexibility when compared with other currencies. As \cite{Zeileis20101696} point out, information on the exchange rate regime of a country is not always fully disclosed by the corresponding central bank. Hence, data driven methods such as linear regression became popular to classify the exchange rate regime in operation. \cite{Zeileis20101696} suggest that a vanishing error variance can be interpreted as a fixed currency regime, while a small or large error variance can indicate a pegged or floating regime respectively. This is illustrating that the error variance is an important quantity when looking for changes in the exchange rate regime. As such changes are often caused by policy interventions, tests for sudden breaks (rather than smooth transitions) are of reasonable interest. 

We consider the exchange rates of the Chinese Yuan Renminbi (CNY) regressed on the exchange rates of the US Dollar (USD). The reason to do so is that China decided to give up on a fixed exchange rate to the US dollar in 2005. More precisely,
we consider 251 data points which are the daily log-difference returns from July 26nd, 2005 to July 25nd, 2006 
of the CNY and USD each with respect to the Swiss franc (CHF) as numeraire currency. This is the first year of observations of a data set considered by \cite{Zeileis20101696} as well as \cite{Kirch20181579}. Both studies use a linear regression model and a basket of four currencies as regressors, namely the USD, Japanese yen (JPY), Euro (EUR) and the British Pound (GBP). However, the results of \cite{Zeileis20101696} indicate nearly vanishing regression coefficients for the JPY, the EUR and the GBP over the whole investigated time period  from July 26nd, 2005 to July 31st, 2009.   

We first apply the bootstrap test by \cite{Mohr2019} to test for changes in the unknown regression function. With a p-value of $90\%$  it suggests a stable regression function. 

Secondly, we apply our test based on $T_{n1}$ 
using the $95\%$-quantile of the limiting distribution $T$ as critical values. The test clearly rejects the null with a p-value smaller than $0.001 \%$, indicating a change in the conditional variance function. The possible change point can be estimated by $ \mbox{argmax}_{s\in[0,1]}(\sup_{z\in\R}|\hat{T}_n(s,z)|)$  and suggests a change of the exchange rate regime in March 3rd, 2006 which is consistent with the results of \cite{Zeileis20101696}. Figure \ref{fig:Vis_exchange} shows the cumulative sum, $\sup_{z\in\R}|\hat{T}_n(\cdot,z)|$ (top plot), as well as the exchange rates of the CNY plotted against the time (bottom plot). The green dashed line is indicating the critical value while the red dashed line corresponds to the estimated change point. 

Note that applying the tests to the full data set, no change in the regression function is detected (p-value $16\%$), but a change in the variance is clearly detected (p-value smaller than $0.001\%$). However, as the data set is rather large and from the findings of \cite{Zeileis20101696} we expect more than one change in the variance when looking at the full set of observations, which makes the estimation of possible changes more complicated (see also section \ref{sec:remarks}).   

\begin{figure}[htbp]
  \centering
     \includegraphics[width=0.8\textwidth]{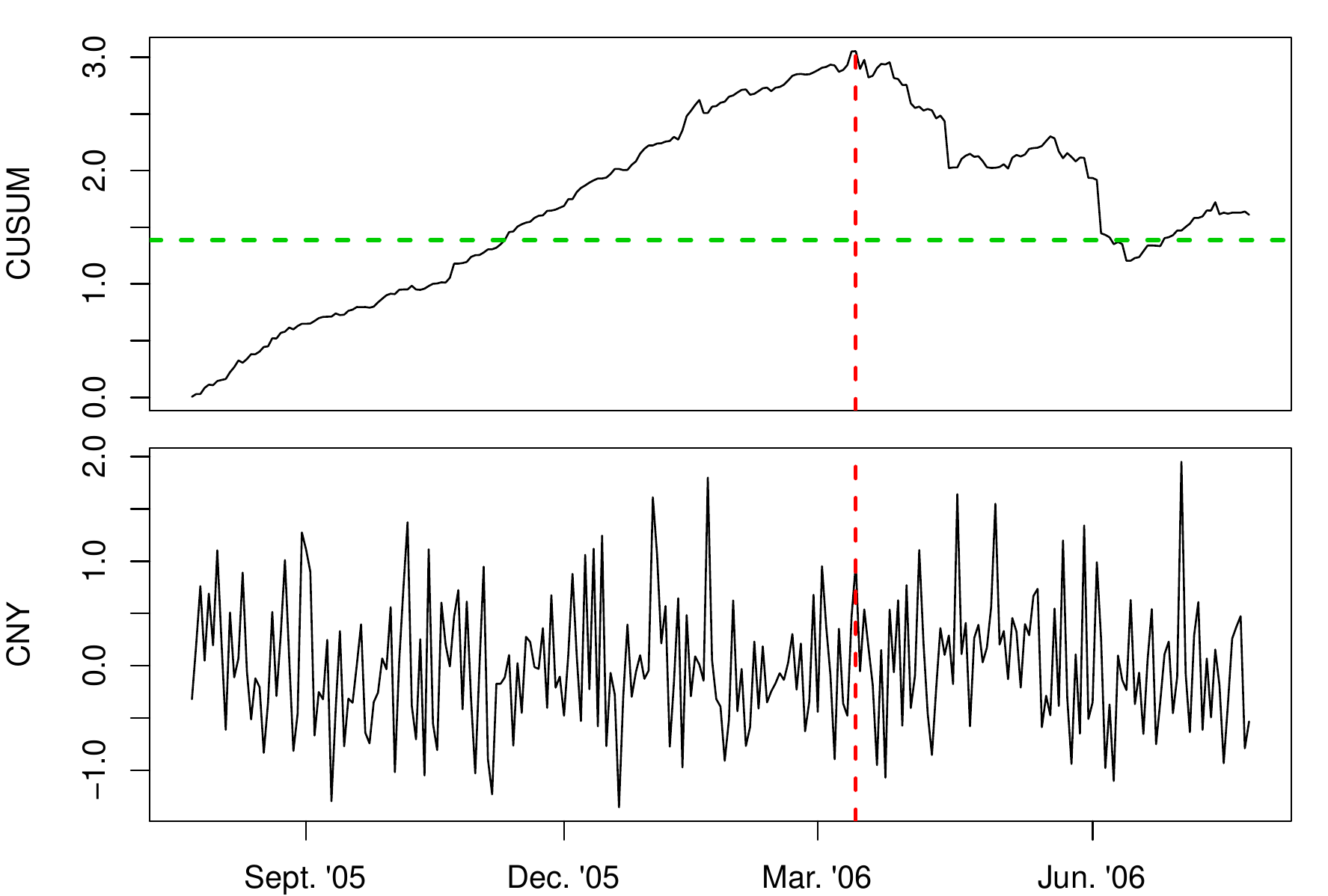}
  \caption{\small Exchange rate data: CUSUM and time series}
  \label{fig:Vis_exchange}
\end{figure}

\section{Concluding remarks} \label{sec:remarks}

This paper closes a gap in the change point testing theory for nonparametric time series models. Assume that one already has accepted that there is no change in the (nonparametric) regression function, but one suspects a change in the (nonparametric) volatility function. In such a case the new test gives a valid procedure. 
To the best knowledge of the authors the new test is the first that can be applied to (nonparametric) autoregressive models (no assumption of bounded support of the covariates) and is consistent against general alternatives of a change point in the variance function. 

Under the assumption that only one change occurs, an estimator for the change point is given by $\mbox{argmax}_{s\in[0,1]}(\sup_{\bm{z}\in\R^d}|\hat{T}_n(s,\bm{z})|)$. Asymptotic properties of this estimator will be considered in future research. If more than one change occurs it might be necessary to modify this estimator. For instance \cite{Fryzlewicz20142243} proposes a wild binary segmentation procedure for the estimation of multiple changes in a simple piecewise-constant signal model, which possibly can be adapted to our setting.

For our theoretical result Theorem \ref{mainth} we need stationarity under the null. However, if there are no changes in both regression function $m$ and variance function $\si^2$, there still could be a change in the error distribution of $\ep_t$. In this case, a bootstrap test similar to the wild bootstrap proposal of \cite{Mohr2019} can be conducted that is sensible to changes in the variance function but not to changes in the error distribution. If both tests of \cite{Mohr2019} and the bootstrap version of the test at hand do not indicate a change in the regression and variance function respectively, the procedure of \cite{Selk2013770} can be used to detect changes in the error distribution.

\appendix

\section{Assumptions}

\begin{enumerate}
\item[\textbf{(G)}] Let $(Y_t,\bm{X}_t)_{t\in\Z}$ be strictly stationary and $\alpha$-mixing with mixing coefficient $\alpha(\cdot)$ such that $\alpha(t)=O(a^{-t})$
for some $a\in(1,\infty)$. 
%
\item[\textbf{($\bm{\xi}$)}] For $\xi_t:=U_t^2-\si^2(\bm{X}_t)$ let there exist some $\gamma>0$ and some even $Q>(d+1)(2+\gamma)$ such that $E[\xi_t|\cF^t]=0$, where $\cF^t=\sigma(U_{j-1},\bm{X}_j:j\le t)$, $E[\xi_t^2|\bm{X}_t]=\tau^2(\bm{X}_t)$ and
 $E[|\xi_t|^{Q\frac{2+\gamma}{2}}|\bm{X}_t]\le c(\bm{X}_t)^Q$ a.s.\ for all $t\in\Z$, for some functions $c, \tau^2: \R^d\to\R$ with 
$\int \bar{c}(\bm{u})f(\bm{u})d(\bm{u})\le M_1$ for some $M_1<\infty$ and $\bar{c}(\bm{u})=\max\left\{\tau^2(\bm{u}),c(\bm{u})^2,\dots,c(\bm{u})^Q\right\}$. 
\item[\textbf{($\bm{\si}$)}] For $Q$, $\gamma$ from assumption \textbf{($\bm{\xi}$)} let $\int |\si^2(\bm{u})|^{Q\frac{2+\gamma}{2}}f(\bm{u})d(\bm{u})\le M_2$ for some $M_2<\infty$.
%
\item[\textbf{(M)}] For some $b>2$ let $E[|Y_1|^{2b}]<\infty$ and let  $\bm{X}_1$ be absolutely continuous with density function $f:\R^d\to\R$ that satisfies $\sup_{\bm{x}\in\R^d}E[|Y_1|^{2b}|\bm{X}_0=\bm{x}]f(\bm{x})<\infty$ and $\sup_{\bm{x}\in\R^d}f(\bm{x})<\infty$. Let there exist some $ j^*<\infty$ such that  $\sup_{\bm{x}_1,\bm{x}_j}E[Y_1^2Y_j^2|\bm{X}_1=\bm{x}_1,\bm{X}_j=\bm{x}_j]f_{1j}(\bm{x}_1,\bm{x}_j)<\infty$ for all $ j\ge j^*$, where $f_{1j}$ is the density function of $(\bm{X}_1,\bm{X}_j)$.
\label{page:(J)}
\item[\textbf{(J)}] Let $(c_n)_{n\in\N}$ be a positive sequence of real numbers satisfying $c_n\to \infty$ and $c_n=O((\log{n})^{1/d})$ and let $\bm{J}_n=[-c_n,c_n]^d$.
\label{page:(F)}
\item[\textbf{(F1)}] For some $C<\infty$ and $c_n$ from assumption \textbf{(J)} let $\bm{I}_n=[-c_n-Ch_n,c_n+Ch_n]^d$, where $h_n$ is from assumption \textbf{(B1)} and \textbf{(B2)} and let $\delta_n^{-1}=\inf_{\bm{x}\in \bm{J}_n}f(\bm{x})>0$ for all $n\in\N$. Further, let for some $r,l\in\N$ and for all $n\in\N$
\begin{eqnarray*}
p_n&=&\max\limits_{\substack{\bm{k}\in\N_0^d\\1\le |\bm{k}|\le l+1+r}}\sup\limits_{\bm{x}\in \bm{I}_n}|D^{\bm{k}}f(\bm{x})|<\infty\\
0<q_n&=&\max\left\{\max\limits_{\substack{\bm{k}\in\N_0^d\\0\le |\bm{k}|\le l+1+r}}\sup\limits_{\bm{x}\in \bm{I}_n}|D^{\bm{k}}m(\bm{x})|,\max\limits_{\substack{\bm{k}\in\N_0^d\\0\le |\bm{k}|\le l+1+r}}\sup\limits_{\bm{x}\in \bm{I}_n}|D^{\bm{k}}\si(\bm{x})|\right\}<\infty,
\end{eqnarray*}
where $|\bm{i}|=\sum_{j=1}^{d}i_j$ and $D^{\bm{i}}=\frac{\partial^{|\bm{i}|}}{\partial x_1^{i_1}\dots\partial x_d^{i_d}}$ for $\bm{i}=(i_1,\dots,i_d)\in\N_0^d$.
\item[\textbf{(F2)}] For $q_n$ from assumption \textbf{(F1)}, $c_n$ from assumption \textbf{(J)} and $C$ from assumption \textbf{(K)}, let for all $\bm{k}\in\N_0^d$ with $|\bm{k}|=2$, 
\[\max\left\{\sup_{\bm{x}\in[-c_n-2h_nC,c_n+2h_nC]^d}\left|D^{\bm{k}}m(\bm{x})\right|,\sup_{\bm{x}\in[-c_n-2h_nC,c_n+2h_nC]^d}\left|D^{\bm{k}}\si(\bm{x})\right|\right\}=O(q_n).\]
\item[\textbf{(K)}] Let $K:\R^d\to\R$ be symmetric in each component, $l+1$ times differentiable with $\int_{\R^d}K(\bm{z})d\bm{z}=1$ and compact support $[-C,C]^d$. Additionally,  let $r\ge 2$ and
$\int_{\R^d}K(\bm{z})\bm{z}^{\bm{k}}d\bm{z}=0$  for all $\bm{k}\in\N_0^d$  with $1\le |\bm{k}|\le r-1$,  where $\bm{z}^{\bm{k}}=z_1^{k_1}\cdots z_d^{k_d}$.
For all $L\in\{K\}\cup\{D^{\bm{k}}K: \bm{k}\in\N_0^d \text{ with } 1\le |\bm{k}|\le l+1\}$ let  $|L(\bm{u})|<\infty$ for all $\bm{u}\in\R^d$ and $|L(\bm{u})-L(\bm{u'})|\le \Lambda \|\bm{u}-\bm{u'}\|$ for some $\Lambda<\infty$ and for all $\bm{u},\bm{u'}\in\R^d$. (Here, $r,l$ and $C$ are from assumption \textbf{(F1)}.)
\item[\textbf{(B1)}]
For $\delta_n,p_n,q_n$ and $r,l$ from assumption \textbf{(F1)} let 
\begin{align*}
\left(\sqrt{\frac{\log n}{nh_n^{d+2(l+1)}}}+h_n^{r}p_n\right)p_n^{l+1}\delta_n^{l+2}=O(1),
\end{align*}
and for some $\eta\in(0,1)$ let
\begin{align*}
\left(\sqrt{\frac{\log n}{nh_n^{d+2(l+1)}}}+h_n^{r}p_n\right)p_n^{l+\eta}q_n^2\delta_n^{l+1+\eta}=o(1).
\end{align*}
\item[\textbf{(B2)}] For $l, p_n, q_n, \delta_n$ from assumption \textbf{(F1)} and $\eta$ from assumption \textbf{(B1)} let
$$\frac{(\log n)^{3+\frac{d}{l+\eta}}}{\sqrt{n^{1-\frac{d}{l+\eta}}h_n^d}}q_n^3\delta_n^2=o(1), \;
\frac{\log {h_n}}{\sqrt{nh_n^d}}=o(1) ,\;
\sqrt{n}h_n^{r}p_nq_n^2=o(1),\;
(\log n)^3h_nq_n^3=o(1)$$
and $\dfrac{(\log n)^{2+\frac{d}{l+\eta}}}{\sqrt{n^{1-\frac{1}{q}-\frac{d}{l+\eta}}}}q_n\delta_n=o(1)$ for $q=Q\frac{2+\gamma}{2}$ with $Q$ and $\gamma$ from assumption \textbf{($\bm{\xi}$)}.
\end{enumerate}

\begin{Remark}\label{rem G}
Assumption \textbf{(G)} is fulfilled by data following causal and stationary ARMA models as they have an MA($\infty$) representation with coefficients that decay exponentially fast (see for instance \cite{Fan2003} Subsection 2.6.1 (iii), p.~69). For more general nonlinear AR-ARCH processes both \cite{Lu19981205} and \cite{Liebscher2005669} give sufficient conditions on regression function, volatility function and the innovations under which the mixing condition in \textbf{(G)} holds. In the linear model
\[Y_t=a_1Y_{t-1}+\dots + a_dY_{t-d}+(b_0+b_1Y_{t-1}^2+\dots+b_dY_{t-d}^2)^{1/2}\ep_t, \ t\in\Z,\]
where $(\ep_t)_t\overset{\text{i.i.d.}}{\sim}\mathcal{N}(0,1)$, the condition in \cite{Lu19981205} simplifies to $(\sum_{i=1}^{d}|a_i|)^2+\sum_{i=1}^{d}b_i<1$.
\end{Remark}

\begin{Remark}
In order to satisfy the first bandwidth assumption in \textbf{(B2)}, a necessary condition is $l+\eta>d$, hence for higher dimensional covariate $\bm{X}_t$, the existence of higher order partial derivatives of $f$ and $m$ is needed. 
In order to satisfy both the first and third bandwidth assumption in \textbf{(B2)} at the same time, depending on the dimension $d$ and the smoothness parameters $l$ and $\eta$, the order of the kernel $r$ needs to be chosen such that $r>\frac{d}{2}\frac{l+\eta}{l+\eta-d}$ holds. As a rule of thumb, one can choose $h_n=O(n^{-k})$ for some
$0<k<\frac{1}{d}-\frac{1}{l+\eta}$ and a kernel, such that $r>\frac{1}{2k}$. That choice satisfies the assumptions given negligible rates for $q_n$ and $\delta_n$.

Further note that the last constraint in \textbf{(B2)} is merely a trade off between existence of moments of $\xi_t$, dimension $d$ and smoothness parameters $l$ and $\eta$. It is satisfied if $q>\frac{l+\eta}{l+\eta-d}$ (given negligible rates for $q_n$ and $\delta_n$).
\end{Remark}


\section{Proofs}
%

\begin{Lemma} \label{Raten kernel}
Under the assumptions of Theorem \ref{mainth} and under $H_0$ the following rates of convergence can be obtained for the kernel estimators $\hat{m}_n$ and $\hat{\si}^2_n$, 
\begin{enumerate}
\item[(i)]
\begin{enumerate}
\item[(a)] $\sup\limits_{\bm{x}\in \bm{J}_n}\left|\hat{m}_n(\bm{x})-m(\bm{x})\right|=O_P\left(\left(\sqrt{\frac{\log{n}}{nh_n^d}}+h_n^rp_n\right)q_n\delta_n\right)$,
\item[(b)] $\sup\limits_{\bm{x}\in \bm{J}_n}\left|D^{\bm{k}}\left(\hat{m}_n(\bm{x})-m(\bm{x})\right)\right|=O_P\left(\left(\sqrt{\frac{\log{n}}{nh_n^{d+2|\bm{k}|}}}+h_n^rp_n\right)p_n^{|\bm{k}|}q_n\delta_n^{|\bm{k}|+1}\right)$ for all $1\le |\bm{k}|\le l+1$,
\item[(c)] $\displaystyle\sup\limits_{\substack{\bm{x},\bm{y}\in \bm{J}_n \\ \bm{x}\neq\bm{y}}}\frac{\left|D^{\bm{k}}\left(\hat{m}_n(\bm{x})-m(\bm{x})\right)-D^{\bm{k}}\left(\hat{m}_n(\bm{y})-m(\bm{y})\right)\right|}{\|\bm{x}-\bm{y}\|^{\eta}}=o_P(1)$ for all $|\bm{k}|= l$, 
\end{enumerate}
\item[(ii)]
\begin{enumerate}
\item[(a)] $\sup\limits_{\bm{x}\in \bm{J}_n}\left|\hat{\si}^2_n(\bm{x})-\si^2(\bm{x})\right|=O_P\left(\left(\sqrt{\frac{\log{n}}{nh_n^d}}+h_n^rp_n\right)q_n^2\delta_n\right)$,
\item[(b)] $\sup\limits_{\bm{x}\in \bm{J}_n}\left|D^{\bm{k}}\left(\hat{\si}^2_n(\bm{x})-\si^2(\bm{x})\right)\right|=O_P\left(\left(\sqrt{\frac{\log{n}}{nh_n^{d+2|\bm{k}|}}}+h_n^rp_n\right)p_n^{|\bm{k}|}q_n^2\delta_n^{|\bm{k}|+1}\right)$ for all $1\le |\bm{k}|\le l+1$,
\item[(c)] $\displaystyle\sup\limits_{\substack{\bm{x},\bm{y}\in \bm{J}_n \\ \bm{x}\neq\bm{y}}}\frac{\left|D^{\bm{k}}\left(\hat{\si}^2_n(\bm{x})-\si^2(\bm{x})\right)-D^{\bm{k}}\left(\hat{\si}^2_n(\bm{y})-\si^2(\bm{y})\right)\right|}{\|\bm{x}-\bm{y}\|^{\eta}}=o_P(1)$ for all $|\bm{k}|= l$. 
\end{enumerate}
\end{enumerate}
\end{Lemma}

Note that the results for the Nadaraya-Watson estimator $\hat{m}_n$ in (i) are also stated in Lemma A.1 in \cite{Mohr2019}. The proof of  Lemma \ref{Raten kernel} is similar to the proof of Theorem 8 in \cite{Hansen2008726} and omitted for the sake of brevity. 

\begin{Lemma}\label{Lemma_01}
Under the assumptions of Theorem \ref{mainth} and under $H_0$ we have uniformly in $s\in[0,1]$ and $\bm{z}\in\R^d$
\begin{align*}
\frac{1}{\sqrt{n}}\sum\limits_{i=1}^{\lf ns \rf}U_i(m(\bm{X}_i)-\hat{m}_n(\bm{X}_i))\om_n(\bm{X}_i)I\{\bm{X}_i\le \bm{z}\}=o_P(1).
\end{align*}
\end{Lemma}

\begin{proof}
For some $l$-times differentiable function $h:\bm{J}_n\to \R$ define the norm
\[\|h\|_{l+\eta}:=\max\limits_{\substack{\bm{k}\in\N_0^d \\ 1\le |\bm{k}|\le l}}\sup\limits_{\bm{x}\in\bm{J}_n}\left|D^{\bm{k}}h(\bm{x})\right|+\max\limits_{\substack{\bm{k}\in\N_0^d \\ |\bm{k}|=l}}\sup\limits_{\substack{\bm{x},\bm{y}\in\bm{J}_n\\\bm{x}\neq \bm{y}}}\frac{\left|D^{\bm{k}}h(\bm{x})-D^{\bm{k}}h(\bm{y})\right|}{\|\bm{x}-\bm{y}\|^{\eta}}\]
and the function class $\cH:=\cC_{1,n}^{l+\eta}(\bm{J}_n):=\{h:\bm{J}_n\to\R: \|h\|_{l+\eta}\le 1,\sup_{\bm{x}\in\bm{J}_n}\left|h(\bm{x})\right|\le z_n(\log n)^{1/2}\}$ with $z_n:=q_n\delta_n ((\log n)/(nh_n^d))^{1/2}$.
The third bandwidth condition in \textbf{(B2)} implies
\begin{align*}
\left(\sqrt{\frac{\log n}{nh_n^d}}+h_n^rp_n\right)q_n\delta_n=O\left(\sqrt{\frac{\log n}{nh_n^d}}q_n\delta_n\right)
\end{align*}
and thus Lemma \ref{Raten kernel} (i) implies that
 $P(\hat{h}_n\in \cC_{1,n}^{l+\eta}(\bm{J}_n))\to 1$ as $n\to\infty$ holds for  $\hat{h}_n(\bm{x})=(m(\bm{x})-\hat{m}_n(\bm{x}))\om_n(\bm{x})$. It is then sufficient to consider $n^{-1/2}\sum_{i=1}^{\lf ns \rf}h(\bm{X}_i)U_iI\{\bm{X}_i\le \bm{z}\}$ for $s\in[0,1]$, $\bm{z}\in\R^d$ and $h\in\cH$.
%
Furthermore, using \textbf{($\bm{\xi}$)} and \textbf{($\bm{\si}$)} it can be shown that for $q:=Q\frac{2+\gamma}{2}>2$
\begin{align*}
&\frac{1}{\sqrt{n}}\sum\limits_{i=1}^{\lf ns \rf}h(\bm{X}_i)U_iI\{\bm{X}_i\le \bm{z}\}\\
=&\frac{1}{\sqrt{n}}\sum\limits_{i=1}^{\lf ns \rf}\left(h(\bm{X}_i)U_iI\{|U_i|\le n^{1/q}\}I\{\bm{X}_i\le \bm{z}\}-E[h(\bm{X}_i)U_iI\{|U_i|\le n^{1/q}\}I\{\bm{X}_i\le \bm{z}\}]\right)\\
&+o_P(1)
\end{align*}
holds uniformly in $s\in[0,1]$ and $\bm{z}\in\R^d$. Defining the function class $\cF:=\{(u,\bm{x})\mapsto uI\{|u|\le n^{1/q}\}I\{\bm{x}\le\bm{z}\}:\bm{z}\in\R^d\}$ and imposing $(U_1,\bm{X}_1)\sim P$, the assertion then follows if we show
\begin{align*}
\sup\limits_{s\in[0,1]}\sup\limits_{\ph\in\cF}\sup\limits_{h\in\cH}\bigg|\frac{1}{\sqrt{n}}\sum\limits_{i=1}^{\lf ns\rf}\big(h(\bm{X}_i)\ph(U_i,\bm{X}_i)-\int h\ph dP\big)\bigg|=o_P(1).
\end{align*}
To this end let $\ep_{n1}=n^{-1/2}n^{-1/q}$, $\ep_{n2}=n^{-1/2}$ and $\ep_{n3}=n^{-1/2}/(\log n)$ and let further $0=s_1<\dots<s_{K_n}=1$ partition $[0,1]$ in intervals of length $2\ep_{n1}$ such that $K_n=O(\ep_{n1}^{-1})$. Furthermore, we use the bracketing numbers $J_n:=N_{[~]}\left(\ep_{n2},\cF,\|\cdot\|_{L_2(P)}\right)$ and $M_n:=N_{[~]}\left(\ep_{n3},\cH,\|\cdot\|_{\infty}\right)$, where $\|\cdot\|_{\infty}$ is the supremum norm on $\bm{J}_n$. Let $[\ph_1^{l},\ph_1^{u}],\dots,[\ph_{J_n}^{l},\ph_{J_n}^{u}]$ denote the brackets needed to cover $\cF$. Let furthermore $[h_1^{l},h_1^{u}],\dots,[h_{M_n}^{l},h_{M_n}^{u}]$ define the brackets needed to cover $\cH$. It can be shown that $J_n=O\left(\ep_{n2}^{-2d}\right)$ and $M_{n}=O(\exp(c_n^d\ep_{n3}^{-d/(l+\eta)}))$ and further
\begin{align*}
&\sup\limits_{s\in[0,1]}\sup\limits_{\ph\in\cF}\sup\limits_{h\in\cF}\bigg|\frac{1}{\sqrt{n}}\sum\limits_{i=1}^{\lf ns \rf}\big(h(\bm{X}_i)\ph(U_i,\bm{X}_i)-\int h\ph dP\big)\bigg|\\
&\le\max\limits_{\substack{\substack{1\le k\le K_n\\1\le j\le J_n}\\1\le m\le M_n}}\sup\limits_{\ph\in[\ph_j^{l},\ph_j^u]}\sup\limits_{h\in[h_m^l,h_m^u]}\bigg|\frac{1}{\sqrt{n}}\sum\limits_{i=1}^{\lf ns_k \rf}\big(h(\bm{X}_i)\ph(U_i,\bm{X}_i)-\int h\ph dP\big)\bigg|+o_P(1)\\
&\le \max\limits_{\substack{\substack{1\le k\le K_n\\1\le j\le J_n}\\1\le m\le M_n}}\bigg\{
\bigg|\frac{1}{\sqrt{n}}\sum\limits_{i=1}^{\lf ns_k\rf}\big(h_m^u(\bm{X}_i)\ph_j^u(U_i,\bm{X}_i)I\{h_m^u(\bm{X}_i)\ph_j^u(U_i,\bm{X}_i)\ge 0\}-\int h_m^u\ph_j^uI\{h_m^u\ph_j^u\ge 0\}dP\big)\bigg|,\\
&\hspace{1cm}
\bigg|\frac{1}{\sqrt{n}}\sum\limits_{i=1}^{\lf ns_k\rf}\big(h_m^l(\bm{X}_i)\ph_j^l(U_i,\bm{X}_i)I\{h_m^l(\bm{X}_i)\ph_j^l(U_i,\bm{X}_i)< 0\}-\int h_m^l\ph_j^lI\{h_m^l\ph_j^l< 0\}dP\big)\bigg|,\\
&\hspace{1cm}
\bigg|\frac{1}{\sqrt{n}}\sum\limits_{i=1}^{\lf ns_k\rf}\big(h_m^l(\bm{X}_i)\ph_j^l(U_i,\bm{X}_i)I\{h_m^l(\bm{X}_i)\ph_j^l(U_i,\bm{X}_i)\ge 0\}-\int h_m^l\ph_j^lI\{h_m^l\ph_j^l\ge 0\}dP\big)\bigg|,\\
&\hspace{1cm}
\bigg|\frac{1}{\sqrt{n}}\sum\limits_{i=1}^{\lf ns_k\rf}\big(h_m^u(\bm{X}_i)\ph_j^u(U_i,\bm{X}_i)I\{h_m^u(\bm{X}_i)\ph_j^u(U_i,\bm{X}_i)< 0\}-\int h_m^u\ph_j^uI\{h_m^u\ph_j^u< 0\}dP\big)\bigg|
\bigg\}\\
&+o_P(1).
\end{align*}
In what follows we only consider the first line on the right hand side, while the other ones can be treated similarly. We apply Theorem 2.1 of \cite{Liebscher199669} to the random variable (for $m,j,k$ fixed)
\[Z_{i}:=\left(h_m^u(\bm{X}_i)\ph_j^u(U_i,\bm{X}_i)I\{h_m^u(\bm{X}_i)\ph_j^u(U_i,\bm{X}_i)\ge 0\}-\int h_m^u\ph_j^uI\{h_m^u\ph_j^u\ge 0\}dP\right)I\left\{\tfrac{i}{n}\le s_k\right\}.\]
The mixing coefficient of $\{Z_t:1\le t\le n\}$ can be bounded by the mixing coefficient of $\{(U_t,\bm{X}_t):t\in\Z\}$ due to \cite{Bradley1985165}, Section 2, remark (iv). Further, the variables are centered and have a bound of order $O(z_n(\log n)^{1/2}n^{1/q})$. Applying Theorem 2.1 to $\sum_{i=1}^{n}Z_i$ yields for all $\epsilon>0$ and $n\in\N$ large enough
\begin{align*}
&P\bigg(\max\limits_{\substack{\substack{1\le k\le K_n\\1\le j\le J_n}\\1\le m\le M_n}}\bigg|\frac{1}{\sqrt{n}}\sum\limits_{i=1}^{\lf ns_k\rf}\Big(h_m^u(\bm{X}_i)\ph_j^u(U_i,\bm{X}_i)I\{h_m^u(\bm{X}_i)\ph_j^u(U_i,\bm{X}_i)\ge 0\}\\
&\qquad\qquad\qquad\qquad{}-\int h_m^u\ph_j^uI\{h_m^u\ph_j^u\ge 0\}dP\Big)\bigg|>\epsilon\bigg)\\
\le&\hspace{-0.2cm}\sum\limits_{\substack{\substack{1\le k\le K_n\\1\le j\le J_n}\\1\le m\le M_n}}\hspace{-0.3cm}P\bigg(\bigg|\sum\limits_{i=1}^{\lf ns_k\rf}\Big(h_m^u(\bm{X}_i)\ph_j^u(U_i,\bm{X}_i)I\{h_m^u(\bm{X}_i)\ph_j^u(U_i,\bm{X}_i)\ge 0\}\\
&\qquad\qquad\qquad\qquad{}-\int h_m^u\ph_j^uI\{h_m^u\ph_j^u\ge 0\}dP\Big)\bigg|>\sqrt{n}\epsilon\bigg)\\
\le& K_nJ_nM_n4\exp\left(-\frac{n\epsilon^2}{64n \lf (nh_n^d)^{1/2}\rf z_n^2\log(n)+\frac{8}{3}n^{1/2}\epsilon \lf (nh_n^d)^{1/2}\rf z_n\log(n)^{1/2}n^{1/q}}\right)\\
&+K_nJ_nM_n4\frac{n}{\lf (nh_n^d)^{1/2}\rf}\alpha\left(\lf (nh_n^d)^{1/2}\rf\right)\\
=&o(1),
\end{align*}
where the first, second and last bandwidth constraint in \textbf{(B2)} were used in the last equality. Details are omitted for the sake of brevity. 

\end{proof}

\begin{Lemma}\label{Lemma_02}
Under the assumptions of Theorem \ref{mainth} and under $H_0$ we have uniformly in $s\in[0,1]$ and $\bm{z}\in\R^d$
\begin{align*}
\frac{1}{\sqrt{n}}\sum\limits_{i=1}^{\lf ns \rf}\left(U_i^2-\hat{\si}^2_n(\bm{X}_i)\right)\om_n(\bm{X}_i)I\{\bm{X}_i\le \bm{z}\}= T _n(s,\bm{z})-s T _n(1,\bm{z})+o_P(1).
\end{align*}
\end{Lemma}

Note that the proof of Lemma \ref{Lemma_02} is similar to the proof of Theorem 3.1 (i) in \cite{Mohr2019}. It will only be sketched for the sake of brevity.

\begin{proof}
Using $\xi_t=U_t^2-\si^2(\bm{X}_t)$ under $H_0$, it holds that
\begin{align*}
&\frac{1}{\sqrt{n}}\sum\limits_{i=1}^{\lf ns \rf}\left(U_i^2-\hat{\si}^2_n(\bm{X}_i)\right)\om_n(\bm{X}_i)I\{\bm{X}_i\le \bm{z}\}\\
&=\frac{1}{\sqrt{n}}\sum\limits_{i=1}^{\lf ns \rf}\xi_i\om_n(\bm{X}_i)I\{\bm{X}_i\le \bm{z}\}
+\frac{1}{\sqrt{n}}\sum\limits_{i=1}^{\lf ns \rf}\left(\si^2(\bm{X}_i)-\hat{\si}^2_n(\bm{X}_i)\right)\om_n(\bm{X}_i)I\{\bm{X}_i\le \bm{z}\}.
\end{align*}
By strict stationarity of $\{(\xi_t,\bm{X}_t):t\in\Z\}$ and the moment constraints from \textbf{($\bm{\xi}$)} we deduce that uniformly in $s\in[0,1]$ and $\bm{z}\in\R^d$
\[\frac{1}{\sqrt{n}}\sum\limits_{i=1}^{\lf ns \rf}\xi_i\om_n(\bm{X}_i)I\{\bm{X}_i\le \bm{z}\}= T _n(s,\bm{z})+o_P(1).\]
Making use of the uniform convergence rates of $\hat{\si}^2_n$ stated in Lemma \ref{Raten kernel} (ii) we furthermore obtain  
\begin{align*}
&\frac{1}{\sqrt{n}}\sum\limits_{i=1}^{\lf ns \rf}\left(\si^2(\bm{X}_i)-\hat{\si}^2_n(\bm{X}_i)\right)\om_n(\bm{X}_i)I\{\bm{X}_i\le \bm{z}\}\\
&=s \sqrt{n} \int \left(\si^2(\bm{x})-\hat{\si}^2_n(\bm{x})\right)\om_n(\bm{x})I\{\bm{x}\le \bm{z}\}f(\bm{x})d\bm{x}+o_P(1),
\end{align*}
uniformly in $s\in[0,1]$ and $\bm{z}\in\R^d$. Continuing by inserting the definition of $\hat{\si}^2_n$, using $Y_i=m(\bm{X}_i)+U_i$ and finally $\xi_i=U_i^2-\si^2(\bm{X}_i)$ under $H_0$, it holds that
\begin{align}
&\sqrt{n} \int \left(\si^2(\bm{x})-\hat{\si}^2_n(\bm{x})\right)\om_n(\bm{x})I\{\bm{x}\le \bm{z}\}f(\bm{x})d\bm{x}\notag\\
&=\sqrt{n} \hspace{-0.2cm}\int\hspace{-0.1cm} \left(\si^2(\bm{x})-\frac{1}{n}\sum\limits_{i=1}^{n}K_{h_n}(\bm{x}-\bm{X}_i)(Y_i-\hat{m}_n(\bm{x}))^2\frac{1}{\hat{f}_n(\bm{x})}\right)\om_n(\bm{x})I\{\bm{x}\le \bm{z}\}f(\bm{x})d\bm{x}\notag\\
&=\frac{1}{\sqrt{n}} \sum\limits_{i=1}^{n}\int \left(\si^2(\bm{x})-(Y_i-\hat{m}_n(\bm{x}))^2\right)K_{h_n}(\bm{x}-\bm{X}_i)\om_n(\bm{x})I\{\bm{x}\le \bm{z}\}\frac{f(\bm{x})}{\hat{f}_n(\bm{x})}d\bm{x}\notag\\
&=-\frac{1}{\sqrt{n}}\sum\limits_{i=1}^{n}\xi_i\int K_{h_n}(\bm{x}-\bm{X}_i)\om_n(\bm{x})I\{\bm{x}\le \bm{z}\}\frac{f(\bm{x})}{\hat{f}_n(\bm{x})}d\bm{x}\label{eq:variance_04}\\
&+\frac{1}{\sqrt{n}} \sum\limits_{i=1}^{n}\int \left(\si^2(\bm{x})-\si^2(\bm{X}_i)\right)K_{h_n}(\bm{x}-\bm{X}_i)\om_n(\bm{x})I\{\bm{x}\le \bm{z}\}\frac{f(\bm{x})}{\hat{f}_n(\bm{x})}d\bm{x}\label{eq:variance_05}\\
&+\frac{1}{\sqrt{n}}\sum\limits_{i=1}^{n}\int(m(\bm{X}_i)-\hat{m}_n(\bm{x}))^2 K_{h_n}(\bm{x}-\bm{X}_i)\om_n(\bm{x})I\{\bm{x}\le \bm{z}\}\frac{f(\bm{x})}{\hat{f}_n(\bm{x})}d\bm{x}\label{eq:variance_06}\\
&+\frac{2}{\sqrt{n}}\sum\limits_{i=1}^{n}U_i\int(m(\bm{X}_i)-\hat{m}_n(\bm{x})) K_{h_n}(\bm{x}-\bm{X}_i)\om_n(\bm{x})I\{\bm{x}\le \bm{z}\}\frac{f(\bm{x})}{\hat{f}_n(\bm{x})}d\bm{x}.\label{eq:variance_07}
\end{align}
Concerning \eqref{eq:variance_04} and \eqref{eq:variance_05}, it can be shown that
\begin{align*}
\frac{1}{\sqrt{n}}\sum\limits_{i=1}^{n}\xi_i\int_{(-\bm{\infty},\bm{z}]} K_{h_n}(\bm{x}-\bm{X}_i)\om_n(\bm{x})\frac{f(\bm{x})}{\hat{f}_n(\bm{x})}d\bm{x}
&=\frac{1}{\sqrt{n}}\sum\limits_{i=1}^{n}\xi_i\om_n(\bm{X}_i)I\{\bm{X}_i\le \bm{z}\}+o_P(1),\\
&= T _n(1,\bm{z})+o_P(1),
\end{align*}
and
\begin{align*}
\frac{1}{\sqrt{n}} \sum\limits_{i=1}^{n}\int_{(-\bm{\infty},\bm{z}]} \left(\si^2(\bm{x})-\si^2(\bm{X}_i)\right)K_{h_n}(\bm{x}-\bm{X}_i)\om_n(\bm{x})\frac{f(\bm{x})}{\hat{f}_n(\bm{x})}d\bm{x}=o_P(1),
\end{align*}
uniformly in $\bm{z}\in\R^d$ respectively. Using the uniform rates of convergences of $\hat{m}_n$ from Lemma \ref{Raten kernel} (i) (a), which also hold on the slightly larger set $\bm{I}_n=[-c_n-Ch_n,c_n+Ch_n]^d$, it can be shown that the term \eqref{eq:variance_06} is negligible uniformly in $\bm{z}\in\R^d$. Finally, using similar methods as for the proof of Lemma \ref{Lemma_01}, it can be shown that the term \eqref{eq:variance_07} is as well negligible uniformly in $\bm{z}\in\R^d$. Putting the results together, the assertion of the lemma follows.

\end{proof}

\begin{proof}[Proof of Theorem \ref{mainth}]
The assertion (i) follows by Lemma \ref{Lemma_01} and Lemma \ref{Lemma_02} and by Lemma \ref{Raten kernel} (i) (a) together with the bandwidth constraints as
\begin{align*}
\hat{ T }_n(s,\bm{z})
&=\frac{1}{\sqrt{n}}\sum\limits_{i=1}^{\lf ns \rf}(m(\bm{X}_i)-\hat{m}_n(\bm{X}_i))^2\om_n(\bm{X}_i)I\{\bm{X}_i\le \bm{z}\}\\
&+\frac{2}{\sqrt{n}}\sum\limits_{i=1}^{\lf ns \rf}U_i(m(\bm{X}_i)-\hat{m}_n(\bm{X}_i))\om_n(\bm{X}_i)I\{\bm{X}_i\le \bm{z}\}\\
&+\frac{1}{\sqrt{n}}\sum\limits_{i=1}^{\lf ns \rf}\left(U_i^2-\hat{\si}^2_n(\bm{X}_i)\right)\om_n(\bm{X}_i)I\{\bm{X}_i\le \bm{z}\}.
\end{align*}
For (ii) note that $\{(\xi_t,\bm{X}_t):t\in\Z\}$ is strictly stationary and strongly mixing under $H_0$ and assumption \textbf{(G)}. Denote by $P$ the marginal distribution of $(\xi_1,\bm{X}_1)$. The assertion then follows by an application of 
Corollary 2.7 in \cite{Mohr2017} to the sequential empirical process $\{n^{-1/2}\sum_{i=1}^{\lf ns \rf} (\ph(\xi_i,\bm{X}_i)-\int \ph dP):s\in[0,1],\ph\in\cF\}$ indexed in the function class $\cF:=\{(\xi,\bm{x})\mapsto \xi I\{\bm{x}\le \bm{z}\}:\bm{z}\in\R^d\}$. The conditions that are needed for the asymptotic equicontinuity of the process 
are implied by assumptions $\textbf{(G)}$ and \textbf{($\bm{\xi}$)}. The convergence of the finite dimensional distributions can be shown by applying Corollary 1 in \cite{Rio199535}, which is a central limit theorem for strongly mixing triangular arrays. 

\end{proof}

\bigskip


\bibliographystyle{apa}
\bibliography{mybibfile}

\end{document}